\documentclass[12pt,draft,amsmath,amscd,amsbsy,amssymb]{amsart}
 \relax

\chardef\bslash=`\\ 

\makeatletter \def\verbatim{\interlinepenalty\@M \@verbatim
\leftskip\@totalleftmargin\advance\leftskip1pc
\frenchspacing\@vobeyspaces \@xverbatim} \makeatother \hfuzz1pc

\makeatletter \def\dgt@k{\dg@DX=-3 \dg@DY=2 \dg@SIZE=3}
\makeatother

\makeatletter \def\dgt@kk{\dg@DX=3 \dg@DY=-1 \dg@SIZE=3}

\theoremstyle{plain} \newtheorem{thm}{Theorem}[section]

\theoremstyle{definition}

\usepackage[all]{xy}

\begin{document}

\title[]
{There is no universal proper metric spaces for asymptotic dimension 1}

\author{Mykhailo Zarichnyi}
\address{Department of Mechanics and Mathematics, Ivan Franko National University of Lviv, 1 Universytrtska Str, 79000 Lviv, Ukraine}
\email{zarichnyi@yahoo.com}
\thanks{}


\begin{abstract} Answering a question of Ma, Siegert, and Dydak we show that there is no universal proper metric space for the asymptotic dimension $n\ge1$.
\end{abstract}

\maketitle
\section{Introduction}

The notion of asymptotic dimension is introduced by Gromov \cite{Gr}. A family $\mathcal A$ of subsets of a metric space $X$ is called uniformly bounded if there is $M>0$ such that $\mathrm{diam}(A)\le M$, for every $A\in\mathcal A$. Given $D>0$, we say that a family $\mathcal A$ is $D$-discrete if $d(A,B)=\inf\{d(a,b)\mid a\in A, b\in B\}\ge D$ for every distinct $A,B\in\mathcal A$.

We say that the asymptotic dimension of $X$ is $\le n$ (written $\mathrm{asdim}\, X\le n$) if for every $D>0$ there exists a uniformly bounded cover $\mathcal U$ of $X$ such that $\mathcal U=\cup_{i=0}^n\mathcal U_i$, where every $\mathcal U_i$, $i=0,1,\dots,n$, is $D$-discrete.

Universal spaces for the asymptotic dimensions $\le n$ (we do not provide here a precise definition) are constructed in \cite{Dr} and \cite{BN}.

Recall that a metric space $X$ is proper if every closed ball in it is compact.
It is proved in \cite[Theorem 7.2]{MSD} that there exists a countable proper ultrametric space  $PU$ such that any  proper metric space $X$ of asymptotic dimension 0 coarsely embeds (see the definition below) in $PU$. In other words, there is a universal proper metric space of asymptotic dimension 0. 

Note that universal spaces for asymptotic dimension are constructed in \cite{DZ,BN}.

The following problem is formulated in \cite{MSD}: Given $n \ge 1$ is there a universal space in the class of proper metric spaces of asymptotic dimension at most $n$?

The aim of this note is to provide a negative answer.

\section{Preliminaries}

A metric space $X$ is called geodesic if, for every $x,y\in X$ there is an isometric embedding $\alpha\colon [0,d(x,y)]\to X$ such that $\alpha(0)=x$, $\alpha(d(x,y))=y$.

A map $f\colon X\to Y$ is called asymptotically Lipschitz if there exist $\lambda,s>0$ such that, for any $x,y\in X$,  $d(f(x),f(y))\le \lambda d(x,y)+s$.

A map $f\colon X\to Y$ is called coarsely uniform if there exists a non-increasing function $\phi\colon \mathbb R_+\to\mathbb R_+$ such that $\lim_{t\to\infty}\phi(t)=\infty$ and
$$d_Y(f(x),f(y))\le \phi(d_X(x,y)), \ \ x,y\in X.$$

In \cite{Dr} it is proved that any coarse uniform map defined on a geodesic metric space is asymptotically Lipschitz.

A coarse uniform map $f\colon X\to Y$ is called a coarse embedding if there exists a non-increasing function $\psi\colon \mathbb R_+\to\mathbb R_+$ such that $\lim_{t\to\infty}\psi(t)=\infty$ and
$$d_Y(f(x),f(y))\ge \phi(d_X(x,y)), \ \ x,y\in X.$$

Given $D>0$, we say that a subset $A$ of a metric space $X$ is $D$-discrete if $d(x,y)\ge D$ for all $x,y\in A$ with $x\neq y$.

Recall that a tree is a connected graph without cycles. We regard any connected graph as a metric space endowed with the geodesic metric. Every edge is assumed to be isometric to the unit line segment.

Given a metric space $(X,d)$ and $\alpha>0$, we define $\alpha X$ to be the metric space $(X,\alpha d)$.

By $B_r(x)$ we denote the ball of radius $r>0$ centered at $x$.

\section{Result}

\begin{thm} There is no proper metric space $X$ with the property that every proper metric space of asymptotic dimension $\le1$ admits a coarse embedding in $X$.
\end{thm}
\begin{proof}
Suppose the contrary and let $X$ be such a space. Let $x_0\in X$ be a base point.

For $r>0$, define $$\Phi(r)=\max\{|A|\mid A\subset B_r(x_0)\text{ is a 1-discrete subset}\}.$$ Because of properness of $X$, $\Phi(r)$ is well-defined.

For every $n\in \mathbb N$, define $\Phi_n(r)=\Phi(nr+n)$, for all $r>0$. There exists a nondecreasing function $\Psi\colon [0,\infty)\to[0,\infty)$ satisfying the property: $\lim_{r\to\infty}(\Psi(r)/\Phi_n(r))=\infty$, for any $n\in \mathbb N$.

Let $T$ be a rooted tree with the root $t_0$. For any $n\in\mathbb N$, define $\Xi(n)=|\{t\in T\mid d(t,t_0)=n\}|$.

For every $n\in \mathbb N$, attach $\Psi(n)$ unit segments to $n-1\in\mathbb R_+$ each by one of its endpoints. The obtained rooted tree (the root is $0\in\mathbb R_+$) is denoted by $T_1$. We see that $\Xi(n)\ge \Psi(n)$.

Let $S$ be the tree which is obtained by attaching, for any $n\in\mathbb N$, a copy of $T_n=nT_1$ to $n-1\in\mathbb R_+$ by its root, which we denote by $y_n$. Note that $S$ is a geodesic metric space. For the sake of convenience, any point $k\in T_n\subset S$ will be denoted by $k_{T_n}$.

We are going to show that there is no coarse embedding of $S$ into $X$. Suppose the contrary, let $f\colon S\to X$ be such an embedding. Since $S$ is geodesic, the map $f$ is asymptotically Lipschitz. Let $\lambda , s$ be the corresponding constants from the definition of the asymptotically Lipschitz map.

Since $f$ is a coarse embedding, there exists $k\in\mathbb N$ such that $d_X(f(x),f(y))\ge1$ whenever $d_S(x,y)\ge k$. We obtain $$d_X(x_0, f(n_{T_k}))\le d_X(x_0,f(n_{T_k}))+d_X(f(0_{T_k}),n_{T_k})\le d_X(x_0,f(n_{T_k}))+\lambda kn+s\le mn+m,$$
for a constant $m\in \mathbb N$ large enough.

We conclude that $\Psi(n)\le \Xi(n)\le \Phi (mn+m)$ for all $n\in\mathbb N$, which contradicts to the choice of $\Psi$.

\end{proof}

\section{Remarks}
A metric space is said to have bounded geometry if for every $R > 0$ there exists $C<\infty$ such that every 1-discrete set contained in a ball of radius $R$ is   of cardinality at most $C$.

Given $n \ge 1$, is there a proper metric into which every proper metric space of bounded geometry and of asymptotic dimension at most $n$ can be coarsely embedded?

\end{document}